\documentclass{article}

\usepackage{amssymb}
\usepackage{latexsym}
\usepackage{graphicx}
\usepackage{amsmath}
\usepackage{amsthm}
\usepackage{amsfonts}
\usepackage{amscd}

\theoremstyle{plain}
\newtheorem{theorem}{Theorem}[section]

\newtheorem{lemma}[theorem]{Lemma}
\newtheorem{definition}[theorem]{Definition}
\newtheorem{corollary}[theorem]{Corollary}

\newtheorem{example}[theorem]{Example}
\newtheorem{remark}[theorem]{Remark}
\newtheorem{note}[theorem]{Note}

\begin{document}

\title{On ML-Certificate Linear Constraints for Rank Modulation with
Linear Programming Decoding and its Application to Compact Graphs}

\author{
Manabu Hagiwara\\
Advanced Industrial Science and Technology\\
Central 2, 1-1-1 Umezono, Tsukuba City,\\
Ibaraki, 305-8568, JAPAN\\
Department of Mathematics, University of Hawaii,\\
2565 McCarthy Mall, Honolulu, HI, 96822, USA\\
}


\date{\empty}

\maketitle

\begin{abstract}
Linear constraints for a matrix polytope
 with no fractional vertex
 are investigated as intersecting research among permutation codes,
 rank modulations, and linear programming methods.
By focusing the discussion to the block structure of matrices,
new classes of such polytopes are obtained
from known small polytopes.
This concept,
called ``consolidation'',
 is applied to find a new compact graph
which is known as an approach for the graph isomorphism problem.
Encoding and decoding algorithms
for our new permutation codes
are obtained from existing algorithms for small polytopes.
The minimum distances associated with Kendall-tau distance and the minimum Euclidean distance
of a code obtained by changing the basis of a permutation code may
be larger than the original one.
\end{abstract}


\section{Introduction}
Permutation codes have been proposed
 for the purpose of digital modulation schemes \cite{Slepian}.
Formally speaking,
a permutation code $(G, \mu)$, or its code space, is an orbit $\{ X \mu \mid X \in G \}$,
where $G$ is  a set of permutation matrices\footnote{In some references, $G$ is chosen as a generalized permutation group, e.g., a signed permutation group.}
 and $\mu$ is a Euclidean vector.
The main goal of permutation code is,
``for a given Euclidean vector $\lambda$,
to find an orbit $X \mu$ which minimizes a distance
$ || X \mu - \lambda ||$ over $X \in G$ by an efficient algorithm.''
This problem can be computationally hard if the cardinality of $G$ is huge.

In recent years, study of permutation codes has been one of 
the most exciting research topics
in coding theory.
In 2008, Jiang \textit{et.al.} discovered remarkable
 application of permutation code for flash memory coding \cite{Jiang}.
In 2010, Barg investigated permutation codes and their error-correction
for rank modulation \cite{Barg:codesInPermutationsAndErrorCorrectionForRankModulation}.
IBM researchers, Papandreou \textit{et.al.},
 reported implementation of permutation codes as drift-tolerant
 multilevel phase-change memory
 \cite{Papandreou:driftTolerantMultilevelPhaseChangeMemory}.

Wadayama discovered a novel approach for error-correction of permutation codes
by using a linear programing method
 \cite{Wadayama:LPDecodablePermutationCodesBasedOnLinearlyConstrainedPermutationMatrices}.
He considered the following problem\footnote{
The original problem is
to maximize
$\mathrm{Trace}( \mu \lambda^T  X )$.
It is directly obtained that
$
\mathrm{Trace}( \mu \lambda^T  X )
= \mathrm{Trace}( \lambda^T  X \mu )
= \lambda^T  X \mu
$
}:
\[
\text{maximize } 
\lambda^T  X \mu,
\text{ for fixed Euclidean vectors }
\mu, \lambda
\]
where $X$ is taken over the Birkhoff polytope,
 which consists of doubly stochastic matrices,
or a subset of that polytope.
He showed the fundamental theorem that
 if a doubly stochastic matrix $X_0$ maximizes
 the problem above, then the matrix $X_0$ minimizes the linear programing problem below and vise versa
$$
\text{minimize }
|| X \mu - \lambda ||,
\text{ for a fixed Euclidean vectors }
\mu, \lambda
$$
where $X$ is taken over the Birkhoff polytope
and the distance $|| \cdot ||$ is the Euclidean distance.
The set of vertices of the Birkhoff polytope is equal to the set of permutation matrices.
In other words,
Wadayama's problem is equivalent to the permutation code problem,
in the form of a linear programming (LP) problem.
It implies that we can apply techniques of LP for decoding if $G$ is the set of permutation matrices.
Some readers may have the question:
``Can we apply this approach to a subset of permutation matrices?''
The answer is yes.
In 
 \cite{Wadayama:LPDecodablePermutationCodesBasedOnLinearlyConstrainedPermutationMatrices},
some new classes of permutation codes were proposed
by considering sub-polytopes of the Birkhoff polytope.
However, there is still a problem.
The new classes contain fractional vertices,
in other words, a sub-polytope may contain vertices which are not permutation matrices.
Hence, 
to find a method that yields linear constraints with no fractional points
would be
a meaningful contribution to the intersecting of 
research among permutation codes, rank modulation, and linear programming.

In this paper,
we present a novel technique to construct linear constraints
that have no fractional vertices
by introducing a structure called ``consolidation''.
This technique allows us to focus the discussions on code size, encoding algorithm
and decoding algorithm to local structures.

\section{Obtained ML-Certificate Permutation Codes}
As our contribution,
we obtain permutation codes that are decodable by using a linear programming method,
 have no fractional vertex
and are constructed from technique of this paper.

Throughout this paper,
 $X_{i,j}$ denotes the $(i,j)$th entry of a matrix $X$.
The index of matrices and vectors will start with not $1$ but $0$;
for example, $(v_0, v_1, v_2)$ denotes a three-dimensional vector.
The set of real numbers will be denoted by $\mathbb{R}$
and the set of $n$-by-$n$ matrices over $\mathbb{R}$ shall be denoted
by $\mathrm{M}_n (\mathbb{R})$.

\subsection{Wreath Product}
In this paper, we embed the permutation group $S_{n}$ on $\{0,1,\dots, n-1\}$
 into the set $\mathrm{M}_{n}(\mathbb{R})$ of matrices by the following manner:
for a permutation $\sigma$ in $S_n$,
we define an $n$-by-$n$ matrix $X^{\sigma}$ by 
$ X^{\sigma}_{i,j} := \delta_{j=\sigma(i)},$
where $\delta$ is the Kronecker delta.

Let us recall a wreath product.
\begin{definition}[Wreath Product]
\label{hagiwaraPermutationCodes2011:definition:wreathProduct}
Let $G$ be a set of $\nu$-by-$\nu$ matrices
 and $S_R$ the symmetric group on $\{ 0, 1, \dots, R-1 \}$.
For $g_0, g_1, \dots, g_{R-1} \in G$ and $\sigma \in S_R$,
define an $(\nu R)$-by-$(\nu R)$ permutation matrix $X := (X_{ij})$ by
\begin{equation*}
X_{ij} := \left \{
\begin{array}{ll}
g_i & \text{if } i = \sigma(j), \\
\mathbf{0} & \text{otherwise,}
\end{array}
\right.
\end{equation*} 
for $0 \le i,j < R$.
$X$ shall be denoted by $(\sigma | g_0, g_1, \dots, g_{R-1})$.
The collection of permutation matrices $(\sigma | g_0, g_1, \dots, g_{R-1})$
 is said to be a \textbf{wreath product} of $G$ and $S_R$ and is denoted by $G \wr S_R$.
\end{definition}

\begin{example}\label{compactGraphCopyTheory:example:8matrices}
Let $G$ be a permutation group on $\{0, 1, \dots, \nu-1\}$
and $R:=3$.
Then $G \wr S_3$ consists of matrices:
\begin{eqnarray*}
\left(
\begin{array}{ccc}
g_0 & \mathbf{0} & \mathbf{0} \\
\mathbf{0} & g_1 & \mathbf{0} \\
\mathbf{0} & \mathbf{0} & g_2 \\
\end{array}
\right),
\left(
\begin{array}{ccc}
g_0 & \mathbf{0} & \mathbf{0} \\
\mathbf{0} & \mathbf{0} & g_1 \\
\mathbf{0} & g_2 & \mathbf{0} \\
\end{array}
\right),
\left(
\begin{array}{ccc}
\mathbf{0} & g_0 & \mathbf{0} \\
g_1 & \mathbf{0} & \mathbf{0}\\
\mathbf{0} & \mathbf{0} & g_2 \\
\end{array}
\right),
\end{eqnarray*}
\begin{eqnarray*}
\left(
\begin{array}{ccc}
\mathbf{0} & g_0 & \mathbf{0} \\
\mathbf{0} & \mathbf{0} & g_1 \\
g_2 & \mathbf{0} & \mathbf{0} \\
\end{array}
\right),
\left(
\begin{array}{ccc}
\mathbf{0} & \mathbf{0} & g_0 \\
g_1 & \mathbf{0} & \mathbf{0} \\
\mathbf{0} & g_2 & \mathbf{0} \\
\end{array}
\right),
\left(
\begin{array}{ccc}
\mathbf{0} & \mathbf{0} & g_0 \\
\mathbf{0} & g_1 & \mathbf{0} \\
g_2 & \mathbf{0} & \mathbf{0} \\
\end{array}
\right),
\end{eqnarray*}
where $g_0, g_1, g_2 \in G$.
\qed
\end{example}

\subsection{ML-Certificate LP-Decodable Permutation Codes}
Let $C_n$ denote a cyclic group of order $n$,
$D_{2n}$ a dihedral group of order $2n$,
and $S_n$ a symmetric group of order $n!$.
We consider the groups $C_n, D_{2n}$ and $S_n$ are
sets of permutation matrices of size $n$-by-$n$.

As is mentioned in introduction,
we define a \textbf{permutation code} $(G, \mu)$
 as a pair of a set $G$ of permutation matrices
and a vector $\mu$.
Our argument in this paper does not rely on a choice of $\mu$.
Thus we focus the explanation on which $G$ is obtained by our construction.

\begin{example}
Let $\nu$ and $R$ be positive integers such that $\nu \neq 2, 4$
and $R \ge 2$.
Define $n := \nu R$.

For each $0 \le r < R$,
define a set $G_r$ of $\nu$-by-$\nu$ permutation matrices
 as one of $C_\nu$, $D_{2 \nu}$, and $S_{\nu}$,
and define a set $G_R$ of $R$-by-$R$ permutation matrices
 as one of $C_R$, $D_{2 R}$, and $S_{R}$.
Then we can construct the following set $G$ of 
$\nu R$-by-$\nu R$ permutation matrices:
\[
G :=
\{
(g_R | g_0, g_1, \dots, g_{R-1})
\mid
g_i \in G_i,
 0 \le i < R
\}.
\]
Let $c, d$ and $s$ denote 
the number of
 $C_\nu$, $D_{2 \nu}$, and $S_\nu$
that are chosen for $G_i$ ($0 \le i < R$) respectively.
Similarly,
define $c_R, d_R$ and $s_R$ 
to be
the number of
 $C_R$, $D_{2 R}$, and $S_R$
that are chosen for $G_R$ respectively.
Hence only one of $c_R, d_R, s_R$ is 1 and the others are $0$.
By using this notation,
the cardinality of $G$ is
\[
 \nu^{c} (2 \nu)^{d} (\nu !)^s 
 R^{c_R} (2 R)^{d_R} (R!)^{s_R}.
\]

Previously known examples are
choices
\[(c,d,s,c_R,d_R,s_R)
=(0,2,0,0,0,1),
(0,0,2,0,0,1).
\]
\end{example}

\begin{example}
Let $R$ be a positive integer with $R \ge 2$
and $\nu =2$.
Define $n := 2 R$.
We explain an example of our obtained set of permutation matrices
of size $n$-by-$n$.

For each $0 \le r < R$,
define $G_r$ to be 
either $C_2$ or the unit group,
 consists of only the identity matrix,
and define $G_R$ as one of $C_R$, $D_{2 R}$, and $S_{R}$.
Then we can construct the following set $G$ of permutation matrices:
\[
G :=
\{
(g_R | g_0, g_1, \dots, g_{R-1})
\mid
g_i \in G_i,
 0 \le i < R
\}.
\]
Let $c$ and $u$ denote 
the number of
 $C_2$'s and the unit groups
that
are chosen for $G_i$ ($0 \le i < R$) respectively.
Similarly,
let $c_R, d_R$ and $s_R$ 
denote 
the number of times
 $C_R$, $D_{2 R}$, and $S_R$
are chosen for $G_R$ respectively.
Hence only one of $c_R, d_R, s_R$ is 1 and the others are $0$.
By using this notation,
the cardinality of $G$ is
\[
 2^{c} 
 R^{c_R} (2 R)^{d_R} (R!)^{s_R}.
\]

For ``$c=R$ and $s_R = 1$'',
$G$ becomes a group and is isomorphic to a signed permutation group
 \textit{``as a group''}
whose permutation code has been investigated in
\cite{Slepian,Peterson:reflectionGroupCodesAndTheirDecoding}.
\end{example}

\begin{example}
Let $R$ be positive integers with $R \ge 2$
and $\nu = 4$.
Define $n := 4 R$.
We explain an example of our obtained set of permutation matrices
of size $n$-by-$n$.
Let $P_4$ denote the set of permutations, which are known as
pure involutions in $S_4$ \cite{Wadayama:LPDecodablePermutationCodesBasedOnLinearlyConstrainedPermutationMatrices}.
Then $P_4$ consists of three elements.

For each $0 \le r < R$,
define $G_r$ as one of $C_4$, $D_{8}$, $S_{4}$ and
$P_4$,
and define $G_R$ as one of $C_R$, $D_{2 R}$, and $S_{R}$.
Then we can construct the following set $G$ of permutation matrices:
\[
G :=
\{
(g_R | g_0, g_1, \dots, g_{R-1})
\mid
g_i \in G_i,
 0 \le i < R
\}.
\]
Let $c, d, s$ and $p$ denote 
the number of times
$C_4$, $D_{8}$, $S_4$
and $P_4$
are chosen for $G_i$ ($0 \le i < R$) respectively.
Similarly,
define $c_R, d_R$ and $s_R$ 
to be the number
of times
$C_R$, $D_{2 R}$, and $S_R$
are chosen for $G_R$ respectively.
Hence only one of $c_R, d_R, s_R$ is 1 and the others are $0$.
By using this notation,
the cardinality of $G$ is
\[
 4^{c} 8^{d} 24^s 3^p
 R^{c_R} (2 R)^{d_R} (R!)^{s_R}.
\]
\end{example}

\section{Compactness}
\subsection{Compact Constraints}
\label{hagiwaraPermutationCdoes2012:subsection:CompactConstraints}
\begin{definition}[Linear Constraints]
A \textbf{linear constraint} $l(X)$ for an $n$-by-$n$ matrix
is defined as either
a linear equation on entries of a matrix or a linear inequality on entries of a matrix.

Formally speaking, by regarding an entry $X_{i,j}$ as a variable ($0 \le i,j < n$),
we state either
\[ l(X) : \sum_{0 \le i,j < n} c_{i,j} X_{i,j} = c_0, \]
or
\[ l(X) : \sum_{0 \le i,j < n} c_{i,j} X_{i,j} \ge c_0, \]
for some $c_0, c_{i,j} \in \mathbb{R}$.
The relation $=$ or $\ge$ is uniquely determined by $l(X)$.
In stead of the symbols $=$ and $\ge$,
we use $\trianglerighteq_l$ (or simply $\trianglerighteq$),
e.g.,
\[ l(X) : \sum_{0 \le i,j < n} c_{i,j} X_{i,j} \trianglerighteq_l c_0. \]

If we do not need to clarify the variable $X$ of a linear constant $l(X)$,
we denote it simply by $l$.
\end{definition}

\begin{definition}[Satisfy, $\models$]
Let $\mathcal{L}$ be a set of linear constraints
 for an $n$-by-$n$ matrix.

For $l \in \mathcal{L}$ and $X \in \mathrm{M}_2 (\mathbb{R})$,
if a matrix $X$ satisfies $l$,
we write $X \models l$.
If $X \models l$ for every $l \in \mathcal{L}$,
we write $X \models \mathcal{L}$.
\end{definition}

\begin{note}
\label{hagiwaraPermutationCodes2012:note:doublyStochasticMatrix}
Let $\mathcal{L}_D$ denote the following set of linear constrains defined as
\begin{eqnarray*}
\mathcal{L}_D &:=&
 \{\text{row-sum constraints}\} \\
& & \cup \{\text{column-sum constraints}\} \\
& & \cup \{\text{positivity}\}.
\end{eqnarray*}
Then $\mathcal{L}_D$ is a doubly stochastic constraint.
For clarifying the size $n$ of a matrix,
we may denote $\mathcal{L}_{D}$ by $\mathcal{L}_{D^{(n)} }$.

Let us define a set $\mathrm{DSM}_n$ as
\[
\mathrm{DSM}_n
 :=
\{ X \in \mathrm{M}_n (\mathbb{R}) \mid X \models \mathcal{L}_D^{(n)} \}.
\]
An element of $\mathrm{DSM}_n$ is said to be a \textbf{doubly stochastic matrix}.

The symbols $\mathcal{L}_D$ and $\mathrm{DSM}_n$ are used throughout this paper.
\end{note}

\begin{definition}[Doubly Stochastic Constraint]
A \textbf{doubly stochastic constraint} $\mathcal{L}$ for an $n$-by-$n$ matrix is a set of linear constraints
such that
$X \models \mathcal{L}$ implies $X \models \mathcal{L}_D$.
\end{definition}

\begin{remark}
Since $X^{\sigma} \models \mathcal{L}_D$ in Note \ref{hagiwaraPermutationCodes2012:note:doublyStochasticMatrix}
for a permutation $\sigma$,
we have
\[ S_n \subset \mathrm{DSM}_n. \]
\end{remark}

\begin{definition}[Doubly Stochastic Polytope]
Let $\mathcal{L}$ be a doubly stochastic constraint
 for an $n$-by-$n$ matrix.

The collection of $n$-by-$n$ matrices which satisfy all of linear constraints in $\mathcal{L}$
is denoted by $\mathcal{D}_n [\mathcal{L}]$.
We call $\mathcal{D}_n [ \mathcal{L}]$ a \textbf{doubly stochastic polytope of} $\mathcal{L}$.

For $\mathcal{D} \subset \mathrm{M}_n ( \mathbb{R} )$,
$\mathcal{D}$ is said to be a \textbf{doubly stochastic polytope}
if there exists a doubly stochastic constraint $\mathcal{L}$ such that
$\mathcal{D} = \mathcal{D}_n [ \mathcal{L} ]$.
\end{definition}

\begin{example}[Birkhoff Polytope]
\label{hagiwaraPermutationCodes2012:example:BirkhoffPolytope}
We use the notation $\mathrm{DSM}_n$, instead of $\mathcal{D}_n [\mathcal{L}_D]$
for the doubly stochastic constraint $\mathcal{L}_D$ in Note \ref{hagiwaraPermutationCodes2012:note:doublyStochasticMatrix}.

The polytope $\mathrm{DSM}_n$ is said to be a \textbf{Birkhoff polytope}.
\end{example}

A Birkhoff polytope $\mathrm{DSM}_n$ is an example of doubly stochastic polytope.
Note that any doubly stochastic polytope is a subset of $\mathrm{DSM}_n$.

\begin{definition}[Vertex]
Let $\mathcal{D}$ be a doubly stochastic polytope.

An element $X \in \mathcal{D}$ is said to be a \textbf{vertex}
if there are neither elements $X_0, X_1 \in \mathcal{D}$ with $X_0 \neq X_1$ nor positive numbers $c_0, c_1 \in \mathbb{R}$
such that
\[
 X = c_0 X_0 + c_1 X_1.
\]
We denote the set of vertices for $\mathcal{D}$ by $\mathrm{Ver}( \mathcal{D} )$.
\end{definition}

\begin{definition}[LP-decodable permutation code]
We call a permutation code $(G, \mu)$ 
\textbf{an LP (Linear Programing)-decodable permutation code}
if there exists a doubly stochastic constraint $\mathcal{L}$
such that 
$G = G_\mathcal{L}$,
where $G_\mathcal{L} := \mathrm{Ver}( \mathcal{D}_n [ \mathcal{L} ]) 
\cap S_n$.
\end{definition}
Let us consider the following algorithm
as an error-correcting decoding algorithm for LP-decodable permutation codes.
\begin{definition}[Error-Correcting Decoding Algorithm]
We define an error-correcting decoding algorithm as follows:
\begin{itemize}
\item \textbf{Input:} vectors $\mu, \lambda$, and a set $\mathcal{L}$ of linear constraint,
\item \textbf{Output:} a vector $\lambda_0$,
\item[1.] Solve the following linear programming problem:
\[ \mathrm{max}_{X \models \mathcal{L} } \lambda^T X \mu. \]
\item[2.] For a solution $X_0$, set $\mu_0 := X_0 \mu$.
\item[3.] Output $\mu_0$.
\end{itemize}
\end{definition}

\begin{remark}
It is important that a solution is a vertex of $\mathcal{D}_n [ \mathcal{L} ]$
if the solution exists uniquely.
Therefore we prefer $\mathcal{L}$ such that
 $\mathrm{Ver}( \mathcal{D}_n [ \mathcal{L} ]) \subset S_n$
for permutation codes $( G_\mathcal{L} , \mu )$.
\end{remark}

\begin{definition}[Compact Constraint]
Let $\mathcal{L}$ be a doubly stochastic constraint
 for an $n$-by-$n$ matrix.

We call $\mathcal{L}$ a \textbf{compact constraint}
if 
\begin{itemize}
\item $\mathcal{L}$ consists of a finite number of linear constraints,
\item the doubly stochastic polytope $\mathcal{D}_n [ \mathcal{L} ]$ is a bounded set,
\item the vertex set
satisfies $\mathrm{Ver}( \mathcal{D}_n [ \mathcal{L} ] ) \subset S_n$.
\end{itemize}
\end{definition}

Our primary interest is to find a new class of compact constraints.
To the best of the author's knowledge,
not many compact constraint have been found.
From here, we introduce seven examples of compact constraints.
Two of them are given below and
the others are in Sec \ref{compactGraphCopyTheory:Section:ExampleOfCompactSeedGraph}.

\begin{theorem}[Birkhoff von-Neuman Theorem]
\label{compactGraphCopyTheory:theorem:birkhoffVon-Neumann}
\label{hagiwaraPermutationCodes2012:BirkhoffVonNeumannTheorem}
\[
\mathrm{Ver}(  \mathrm{DSM}_n  ) = S_n.
\]
\end{theorem}
Thus $\mathcal{L}_D$ is compact.

\begin{example}[Pure Involution]
\label{hagiwaraPermutationCodes2012:expl:pureInv}
The following linear constraint $\mathcal{L}_P$ is introduced in
\cite{Wadayama:LPDecodablePermutationCodesBasedOnLinearlyConstrainedPermutationMatrices}:
$\mathcal{L}_P := \mathcal{L}_{D^{(n)}} \cup \{ \sum_{0 \le h < n} X_{hh} = 0,
 X_{ij} - X_{ji} = 0 \text{ for any } 0 \le i,j < n \}.$
It is known that $\mathcal{L}_{P}$ is compact for $n = 2, 4$ but not for $n \ge 6$.
\end{example}

\subsection{Compact Graph}
\label{compactGraphCopyTheory:section:BirkhoffVonNeumannTheorem}
\label{compactGraphCopyTheory:section:graphAutomorphismAndAMultiCopiedUnionGraph}

The notion of compact graph has been introduced 
for the study of the graph isomorphism problem,
 a famous problem in computer science.
Even though the motivation of the study of compact graph seems far from error-correcting codes,
we apply it to permutation codes.

Let $\Gamma := (\{ 0, 1,\dots, n-1 \}, E)$ be a connected graph with its vertex set $\{ 1, 2, \dots n \}$
and its edge set $E$.
Recall that $E \subset \{ 0, 1,\dots, n-1 \}^2$.
Note that, in this paper, $\Gamma$ may be a directed graph or an undirected graph.
Let $A^{\Gamma}$ be the \textbf{adjacency matrix} of $\Gamma$,
\textit{i.e.},
$A^{\Gamma} $ is an $n$-by-$n$ zero-one-matrix over $\mathbb{R}$
and
its $(i,j)$-entry $A^{\Gamma}_{i,j}$ is
\begin{equation*}
A^{\Gamma}_{i,j} := \left \{
\begin{array}{ll}
1 & \text{if }(i,j) \in E, \\
0 & \text{otherwise}
\end{array}
\right.
\end{equation*} 
\begin{definition}[$\mathcal{L}_\Gamma$]
\label{hagiwaraPermutationCodes2012:defn:graphConstraint}
For a graph $\Gamma=( \{ 0, 1, \dots, n-1 \}, E)$,
 we define a \textbf{doubly stochastic constraint}
 $\mathcal{L}_{\Gamma}$ by
\[
\mathcal{L}_{\Gamma}
:=
\mathcal{L}_D
\cup 
\{
 X A^{\Gamma} = A^{\Gamma} X
\}.
\]
Note that $X A^{\Gamma} = A^{\Gamma} X$ defines $n^2$-linear equations
by regarding each entry as an equation.
\end{definition}

\begin{figure}[htbp]
 \includegraphics[width=85mm]{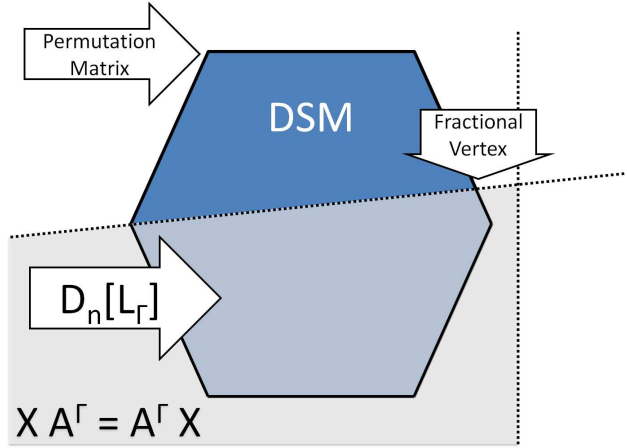}
 \caption{$\mathrm{DSM}$ and $\mathcal{D}_n [ \mathcal{L}_{\Gamma}]$: 
we like to avoid fractional vertices due to the additional equations
$X A^\Gamma = A^\Gamma X$.}
 \label{compactGraphCopyTheory:dsmAndPGamma.eps}
\end{figure}

For a permutation $\sigma$ and a graph $\Gamma$ with $n$-vertices,
we define a graph $\sigma ( \Gamma )$ 
as a graph associated with an adjacency matrix
$X^{\sigma} A_{\Gamma} (X^{\sigma})^{-1}$,
where $X^{\sigma}$ is a permutation matrix associated with $\sigma$
and $A_{\Gamma}$ is the adjacency matrix of $\Gamma$.

\begin{definition}[Automorphism]
Let $\sigma$ be a permutation 
and $\Gamma$ a graph.
Let $X^{\sigma}$ be the permutation matrix associated with $\sigma$.

The permutation $\sigma$ is called an \textbf{automorphism} of $\Gamma$
if
$\sigma (\Gamma) = \Gamma $ holds,
equivalently,
 $X^{\sigma} A^{\Gamma} (X^{\sigma})^{-1} = A^{\Gamma}$ holds,
where $A^{\Gamma}$ is the adjacency matrix of $\Gamma$.
Let $\mathrm{Aut}(\Gamma)$ denote the set of
automorphisms of $\Gamma$.
\end{definition}
It is easy to verify that
\[
 X \in \mathrm{Aut}(\Gamma)
 \iff
 X A^{\Gamma} = A^{\Gamma} X
\]
for a permutation matrix $X$.
Therefore $\mathcal{D}_n [ \mathcal{L}^{\Gamma}]\supset \mathrm{Aut}(\Gamma)$ holds.
By this inclusion
 and Birkhoff von-Neumann theorem (Theorem \ref{compactGraphCopyTheory:theorem:birkhoffVon-Neumann}), 
we have
\[
 \mathrm{Ver} ( \mathcal{D}_n [ \mathcal{L}_{\Gamma}] ) \supset \mathrm{Aut} ( \Gamma )
\]
 for any graph $\Gamma$.

Hence the following are equivalent:
\begin{eqnarray}
 \mathrm{Ver}( \mathcal{D}_n [ \mathcal{L}_{\Gamma}] ) &=& \mathrm{Aut}( \Gamma ), \nonumber \\
 \mathrm{Ver}( \mathcal{D}_n [ \mathcal{L}_{\Gamma}] ) &\subset& \mathrm{Aut}( \Gamma ), \nonumber \\
 \mathrm{Ver}( \mathcal{D}_n [ \mathcal{L}_{\Gamma}] ) &\subset& S_n. \label{hagiwaraPerm2012:eqn:graphAut}
\end{eqnarray}
We present some examples of graphs for which the equality above holds
in Sec.\ref{compactGraphCopyTheory:Section:ExampleOfCompactSeedGraph}.
\begin{definition}[Compact Graph]
Let $\Gamma$ be a (directed or un-directed) graph.
In this paper,
$\Gamma$ is called \textbf{compact}
if
\[
\mathrm{Ver}( \mathcal{D}_n [ \mathcal{L}_{\Gamma}] ) = \mathrm{Aut}( \Gamma ).
\]
\end{definition}

\begin{remark}
The notion of compact graph is introduced by Tinhofer
 \cite{Tinhofer:graphIsomorphismAndTheoremsOfBirkhoffType}.
The original definition of a compact graph is restricted to ``un-directed graphs.''
The idea to allow us to use directed graphs
 is ours.
In Sec.\ref{compactGraphCopyTheory:Section:ExampleOfCompactSeedGraph},
 we shall obtain a new class of directed compact graphs,
 called ``cycles''.
\end{remark}

\subsection{Examples of Compact Graphs}
\label{compactGraphCopyTheory:Section:ExampleOfCompactSeedGraph}

Tinhofer showed that ``any connected tree and any cycle are
compact graphs'' \cite{Tinhofer:graphIsomorphismAndTheoremsOfBirkhoffType}
and
``a union of the same two connected un-directed graph is compact''
in  \cite{Tinhofer:aNoteOnCompactGraphs} (see Remark \ref{compactGraphCopyTheory:remark:doubleCopyTheorem}).
On the other hand, Schreck showed that ``a compact regular graph with prime vertices
must be a circulant graph'' \cite{Schreck:aNoteOnCertainSubpolytopesOfTheAssignmentPolytopeAssociatedWithCirculantGraphs}.
Therefore, it does not seem easy to obtain.

\begin{example}[Complete Graph]
Let $\Gamma = (\{ 0, 1,\dots, n-1 \}, E)$ be a \textit{complete graph},
\textit{i.e.}, $E = \{ (i,j) \in \{ 0, 1, \dots, n-1 \}^2 \mid i \neq j \}$.
Then $\Gamma$ is compact.
Since $\Gamma$ is complete, $\mathrm{Aut}(\Gamma)$ is a symmetric group $S_n$,
 \textit{i.e.}, the set of permutation matrices.
\begin{figure}[htbp]
\begin{center}
 \includegraphics[width=55mm]{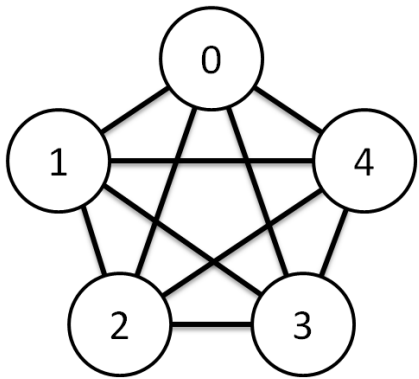}
 \caption{Complete Graph of $n=5$}
\end{center}
\end{figure}
\end{example}

\begin{example}[Tree]
A graph is said to be a \textbf{tree}
if the graph is connected and has no cycle and no loop.
A tree is known as to be compact graph
\cite{Tinhofer:graphIsomorphismAndTheoremsOfBirkhoffType}.
\end{example}

The next examples are examples of trees.
They give automorphism groups which have been investigated
 in \cite{Slepian} and \cite{Peterson:reflectionGroupCodesAndTheirDecoding} under other decoding algorithms.

\begin{example}[Line and Televis\footnote{Televis: a toy consists of two balls and a string which connects the balls}]
\label{hagiwaraPerm2012:expl:lineTelevis}
Let $n$ be a positive integer and $E:=\{ (i,j) \mid i-j = \pm 1  \}$.
Since $\Gamma := ( \{ 0, 1, \dots, n-1 \}, E)$ is a tree,
it is compact.
Then $\mathrm{Aut}( \Gamma )$ is isomorphic to a cyclic group $C_2$ of order $2$.
We call this graph a \textbf{line}.
If $n=2$, we call the graph a \textbf{televis}.

\begin{figure}[htbp]
\begin{center}
 \includegraphics[width=55mm]{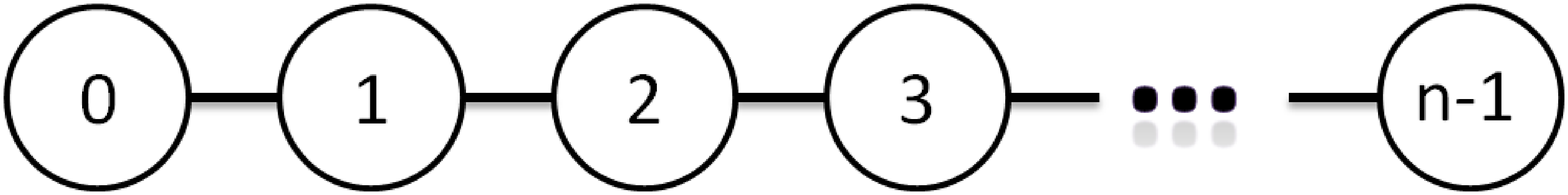}
 \caption{Graph of Type Line}
\end{center}
\end{figure}
\end{example}

\begin{example}[Circle]
Let $n$ be a positive integer and $E:=\{ (i,j) \mid i-j = \pm 1 \pmod{n} \}$.
Then $\Gamma := ( \{ 0, 1,\dots, n-1 \}, E)$ is compact \cite{Tinhofer:graphIsomorphismAndTheoremsOfBirkhoffType}.
The automorphism group $\mathrm{Aut}( \Gamma )$ is known as a dihedral group $\mathrm{D}_{2n}$\footnote{In some references, a dihedral group of degree $n$ is denoted by $D_n$.} of degree $n$,
 or equivalently known as a reflection group of type $I_{n}$ in Humphrey's book \cite{Humphreys:reflectionGroupsAndCoxeterGroups},
since $\# \mathrm{D}_{2n} = 2n$.

\begin{figure}[htbp]
\begin{center}
 \includegraphics[width=70mm]{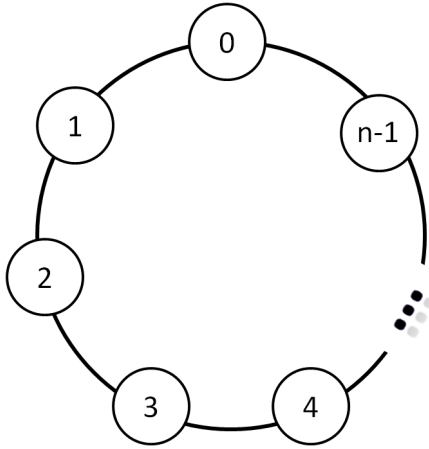}
 \caption{Graph of Type Circle}
\end{center}
\end{figure}
\end{example}

Thus for, we have presented examples of un-directed graphs that satisfy the condition for our main theorem.
From here on, we discuss a ``directed'' graph that satisfies it.

\begin{example}[Cycle]
Let $\Gamma = ( \{0, 1, \dots, n-1 \}, E)$ be a directed cyclic graph,
\textit{i.e.},
$E = \{ (0,1), (1,2), \dots, (n-2, n-1), (n-1,0) \}$.

Then $\Gamma$ is a compact graph.
$\mathrm{Aut}( \Gamma )$ is isomorphic to a cyclic group $C_n$ of order $n$.
\end{example}
The class ``cycle'' has not been known to be a compact graph.
We should present a proof here.
\begin{proof}
It is easy to determine the automorphism group $\mathrm{Aut}(\Gamma)$ for a cyclic graph $\Gamma$.
The group consists of $n$ cyclic permutation $c_0, c_1, \dots, c_{n-1}$,
where $c_v (i) = i+v \pmod{n} $.

Observe $\mathrm{Ver}( \mathcal{D}_n[\mathcal{L}_{\Gamma}] )$.
The equation $A^\Gamma X = X A^\Gamma$ is equivalent to
\[ X_{0,0+v\pmod{n}} = \dots = X_{n-1, n-1+v\pmod{n}},\]
for any $v \in \{0,1, \dots, n-1 \}$.
Therefore,
for any $X \in \mathrm{Ver}( \mathcal{D}_n[\mathcal{L}_{\Gamma}])$,
there exist constants $\lambda_0, \lambda_1, \dots, \lambda_{n-1} \in \mathbb{R}$ such that
\[
X = \sum_{v \in \{0, 1, \dots, n-1 \} } \lambda_v X^{c_{v}},
\]
where $X^{c_v}$ is a permutation matrix associated with a permutation $c_v$.
\begin{figure}[htbp]
\begin{center}
 \includegraphics[width=70mm]{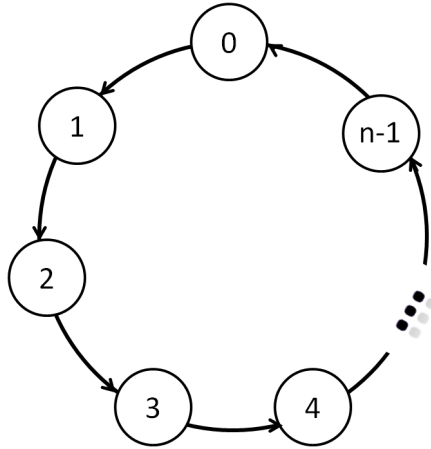}
 \caption{Graph of Type Cycle}
\end{center}
\end{figure}

If $X$ is not a cyclic permutation matrix, 
$X$ must be a linear combination of the cyclic permutation matrices.
Thus the vertices of $\mathcal{D}_n [\mathcal{L}_{\Gamma}]$ are precisely
the cyclic permutation matrices.
\end{proof}
\section{Consolidation}
\subsection{Merged Constraints and Holding Constraints}
\label{hagiwaraPerm2012:subsec:MergedConstaints}
The aim of this section is 
to introduce a novel technique to construct a compact constraint.
\begin{definition}[Homogeneous Constraints]
Let $l$ be a linear constraint for an $n$-by-$n$ matrix.
We call $l$ \textbf{homogeneous} if
the constant term of $l$ is $0$.

Formally speaking,
\[
 l(X) : \sum_{0 \le i,j < n} c_{i,j} X_{i,j} \trianglerighteq_l 0,
\]
for some $c_{i,j} \in \mathbb{R}$.
\end{definition}

\begin{example}
The ``positivity'' is homogeneous
but the ``row-sum constraint'' is not.
Linear constraints $X A^\Gamma = A^\Gamma X$ obtained
from a graph $\Gamma$ are homogeneous.
The linear constraint for pure involution is also homogeneous.
\end{example}

The following constraints are homogeneous too.
\begin{definition}[Weak Row-sum (Column-sum) Constraint]
We call the following $n$-linear constraints \textbf{weak row-sum constraints}:
\[
\sum_{0 \le j < n} X_{i_0, j} = \sum_{0 \le j < n} X_{0, j},
\text{ for any } 0 \le i_0 < n.
\]
Similarly, we call the following $n$-linear constraints \textbf{weak column-sum constraints}:
\[
\sum_{0 \le i < n} X_{i, j_0} = \sum_{0 \le i < n} X_{i, 0},
\text{ for any } 0 \le j_0 < n.
\]
\end{definition}

\begin{lemma}[Constant Sum Property]
\label{hagiwaraPermutationCodes2012:proposition:constantSumProperty}
If a matrix satisfies both the weak row-sum and the weak column-sum
constraints, then any two row-sums or column-sums are equal,
\end{lemma}
\begin{proof}
By weak row-sum constraints,
any row-sums are equal to each other.
So are any column-sums.
Thus any row and column sum is equal to
$\frac{1}{n} \sum_{0 \le i,j < n} X_{i,j}$,
where $n$ is the size of the matrix $X$.
Hence the statement holds.
\end{proof}

\begin{definition}[Quasi-homogeneous Constraint]
A linear constraint $\mathcal{L}$ is said to be a \textbf{quasi-homogeneous constraint}
if
$\mathcal{L}$ consists of
homogeneous constraints, all row-sum constraints and all column-sum constraints.
\end{definition}

For example, $\mathcal{L}_D$ in Note \ref{hagiwaraPermutationCodes2012:note:doublyStochasticMatrix},
$\mathcal{L}_P$ in Expl. \ref{hagiwaraPermutationCodes2012:expl:pureInv}
 and $\mathcal{L}_\Gamma$ in Def. \ref{hagiwaraPermutationCodes2012:defn:graphConstraint}
 are quasi-homogeneous constraints.

\begin{definition}[Merged Constraint]
Let $\mathcal{L}$ be a quasi-homogeneous constraint.
For $\mathcal{L}$, we define another set $\mathcal{L}^\square$ of linear constraints
by replacing row-sum constraints in $\mathcal{L}$ with weak row-sum constraints
and by replacing column-sum constraints in $\mathcal{L}$ with weak column-sum constraints.
We call $\mathcal{L}^\square$ a \textbf{merged constraint} for $\mathcal{L}$.
\end{definition}

\begin{remark}
A merged constraint is homogeneous.
\end{remark}

\begin{example}
A merged constraint
$\mathcal{L}_D^{\square}$
for the constraints $\mathcal{L}_D$ in
 Note \ref{hagiwaraPermutationCodes2012:note:doublyStochasticMatrix}
 consists of three kinds of linear constraints:
\begin{itemize}
\item weak row-sum constraints,
\item weak column-sum constraints,
\item positivity.
\end{itemize}
\end{example}

Let $\nu$ and $R$ be positive integers.
For a $\nu R$-by-$\nu R$ matrix $X$,
we may divide $X$ into $R^2$ block matrices $X^{[r_0, r_1]}$ of size $\nu$-by-$\nu$
via the following relation:
\[
 X^{[r_0, r_1]}_{i,j}
=
 X_{r_0 \nu + i, r_1 \nu + j},
\]
for $0 \le i,j < \nu$ and $0 \le r_0, r_1 < R$.

For example, if $\nu=3$ and $R=2$, we have
\begin{eqnarray*}
& &
\left(
\begin{array}{cccccc}
X_{00} & X_{01} & X_{02} & X_{03} & X_{04} & X_{05}\\
X_{10} & X_{11} & X_{12} & X_{13} & X_{14} & X_{15}\\
X_{20} & X_{21} & X_{22} & X_{23} & X_{24} & X_{25}\\
X_{30} & X_{31} & X_{32} & X_{33} & X_{34} & X_{35}\\
X_{40} & X_{41} & X_{42} & X_{43} & X_{44} & X_{45}\\
X_{50} & X_{51} & X_{52} & X_{53} & X_{54} & X_{55}
\end{array}
\right)\\
&=&
\left(
\begin{array}{cc}
X^{[00]} & X^{[01]}\\
X^{[10]} & X^{[11]}
\end{array}
\right)
\end{eqnarray*}
and
\begin{eqnarray*}
X^{[01]}
=
\left(
\begin{array}{ccc}
X_{03} & X_{04} & X_{05} \\
X_{13} & X_{14} & X_{15} \\
X_{23} & X_{24} & X_{25}
\end{array}
\right).
\end{eqnarray*}
We call $X^{[r_0, r_1]}$ the \textbf{$(r_0,r_1)$th block} of $X$.

\begin{definition}[Holding Constraints]
Let $\mathcal{H}$ be a set of linear constraints
for an $R$-by-$R$ matrix.

For $h(H) \in \mathcal{H}$,
we define a linear constraint $h^{\#}(X)$
 for $\nu R$-by-$\nu R$ matrix
by replacing
$H_{r_0, r_1}$ with $\sum_{0 \le j < \nu} X_{0, j}^{[r_0, r_1]}$.
For $\mathcal{H}$,
we define a set $\mathcal{H}^{\#}$ of linear constraints
for $\nu R$-by-$\nu R$ matrix
as
\[
 \mathcal{H}^{\#} := \{ h^{\#} \mid h \in \mathcal{H} \}.
\]
We call $\mathcal{H}^{\#}$ a \textbf{holding constraint} associated with $\mathcal{H}$
of degree $\nu$.
\end{definition}

\begin{example}
Let us define
 $\mathcal{H} := \{
 h_1(H) : H_{00}+H_{01}=1,
 h_2(H) : H_{00}+H_{10}=1,
 h_3(H) : H_{11} \ge 0
 \}$.
The holding constraint $\mathcal{H}^{\#}$ of degree 3 is
\begin{eqnarray*}
\mathcal{H}^{\#} =\{
 & &\\
 & h_1^{\#}(X) :&
(X^{[00]}_{00}+X^{[00]}_{01}+X^{[00]}_{02}) \\
 & & +
(X^{[01]}_{00}+X^{[01]}_{01}+X^{[01]}_{02})
=1,\\
 & h_2^{\#}(X) :&
(X^{[00]}_{00}+X^{[00]}_{01}+X^{[00]}_{02}) \\
 & & +
(X^{[10]}_{00}+X^{[10]}_{01}+X^{[10]}_{02})
=1,\\
 & h_3^{\#}(X) :&
(X^{[11]}_{00}+X^{[11]}_{01}+X^{[11]}_{02})
\ge 0\\
\}. & &\\
\end{eqnarray*}
\end{example}

\subsection{Consolidation}

The following is a key idea of our construction.
For easy reading,
we recommend to read it with Figure \ref{hagiwaraPermutationCodes:consolidation.eps}.
\begin{definition}[Consolidation]
Let $\mathcal{M}^{[r_0,r_1]}$ be a quasi-homogeneous
 constraint for a $\nu$-by-$\nu$ matrix
 for $0 \le r_0, r_1 < R$.
Let $\mathcal{H}$ be a set of linear constraints for an $R$-by-$R$ matrix.

For $\{ \mathcal{M}^{[r_0,r_1]} \}$ and $\mathcal{H}$,
we define another linear constraint $\mathcal{M} \boxplus \mathcal{H}$ 
for a $\nu R$-by-$\nu R$ matrix
as follows:
\begin{eqnarray*}
\mathcal{M} \boxplus \mathcal{H}
&:=&
\{ m^{[r_0, r_1] \square} (X^{[r_0,r_1]}) \mid \\
& & m^{[r_0,r_1]} \in \mathcal{M}^{[r_0, r_1]},
0 \le r_0, r_1 <R
\} \\
&  & \cup \{ h^{ \# } \mid h \in \mathcal{H} \},
\end{eqnarray*}
where $X$ is a $\nu R$-by-$\nu R$ matrix
and $X^{[r_0, r_1]}$ is the $(r_0,r_1)$th block of $X$
of size $\nu$-by-$\nu$.

We call $\mathcal{M} \boxplus \mathcal{H}$ a \textbf{consolidation} for
$\{ \mathcal{M}^{[r_0,r_1]} \}$ and $\mathcal{H}$.
\end{definition}

\begin{figure}[htbp]
 \includegraphics[width=85mm]{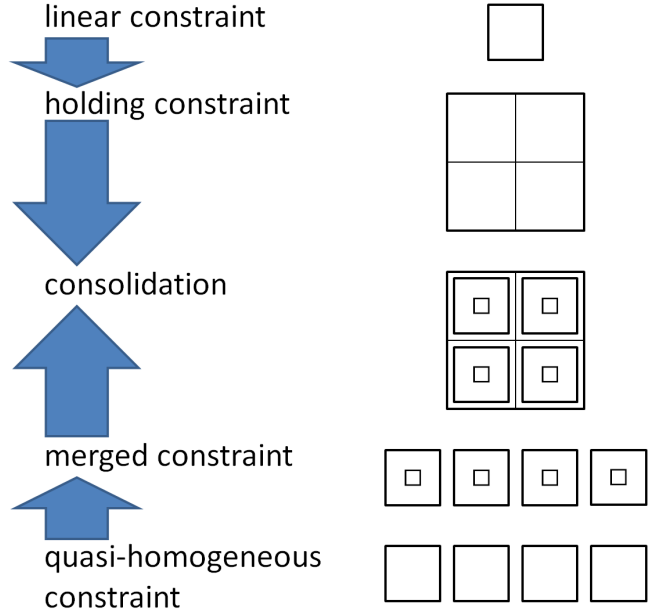}
 \caption{Consolidation}
 \label{hagiwaraPermutationCodes:consolidation.eps}
\end{figure}

\begin{example}
Let $\mathcal{L}_{D^{(2)} }$ be the doubly stochastic constraint in Note \ref{hagiwaraPermutationCodes2012:note:doublyStochasticMatrix}
 for a 2-by-2 matrix.
Put $\mathcal{M}^{[0,0]} := \mathcal{L}_{D^{(2)}}$,
$\mathcal{M}^{[0,1]} := \mathcal{L}_{D^{(2)}}$, and
$\mathcal{M}^{[1,0]} := \mathcal{L}_{D^{(2)}}$.
Put $\mathcal{M}^{[1,1]} := \mathcal{L}_{D^{(2)}} \cup \{ X_{0,0} + X_{1,1} = 0 \}$.
Let $\mathcal{H}$ be the doubly stochastic constraint in Note \ref{hagiwaraPermutationCodes2012:note:doublyStochasticMatrix}
 for a 3-by-3 matrix.

Then the consolidation $\mathcal{M} \boxplus \mathcal{H}$
 is a doubly stochastic constraint for a 6-by-6 matrix.
$\mathcal{M} \boxplus \mathcal{H}$ consists of:
\begin{eqnarray*}
\sum_{0 \le i < 6} X_{i, j_0} &=& 1, \text{ for } 0 \le j_0 < 6,\\
\sum_{0 \le j < 6} X_{i_0, j} &=& 1, \text{ for } 0 \le i_0 < 6,\\
X_{i_0, j_0} &\ge& 0, \text{ for } 0 \le i_0, j_0 < 6,\\
\sum_{0 \le i < 3} X_{2 r_0 + i, 2 r_1 + j_0}
 &=&
\sum_{0 \le j < 3} X_{2 r_0 , 2 r_1 + j}, \text{ for }\\
 & &
 0 \le j_0 < 3,
 0 \le r_0, r_1 < 2,\\
\sum_{0 \le j < 3} X_{2 r_0 + i_0, 2 r_1 + j} 
 &=&
\sum_{0 \le i < 3} X_{2 r_0 +i , 2 r_1 }, \text{ for }\\
 & &
 0 \le i_0 < 3,
 0 \le r_0, r_1 < 2,\\
X_{3,3} + X_{4,4} + X_{5,5} &=& 0.
\end{eqnarray*}
\end{example}

The main contribution of this section is the following:
\begin{theorem}
\label{hagiwaraPermutationCodes2012:thm:consolidationConstruction}
Let $\mathcal{M}^{[r_0,r_1]}$ be a quasi-homogeneous constraint for a $\nu$-by-$\nu$ matrix
 for $0 \le r_0, r_1 < R$.
Let $\mathcal{H}$ be a doubly stochastic constraint for an $R$-by-$R$ matrix.

If $\mathcal{M}^{[r_0, r_1]}$ for any $r_0, r_1$
and
$\mathcal{H}$
 are compact,
we have the following:
\begin{itemize}
\item[1)] $\mathrm{Ver}( \mathcal{D}_{\nu R} [ \mathcal{M} \boxplus \mathcal{H}])
 = \{ (H_{r_0, r_1} X^{[r_0, r_1]} ) \mid
 H \in \mathrm{Ver}( \mathcal{D}_R [\mathcal{H}]),
 X^{[r_0,r_1]} \in \mathrm{Ver}( \mathcal{D}_\nu [ \mathcal{M}^{[r_0,r_1]} ]),
 0 \le r_0, r_1 < R \}.$
\item[2)] the consolidation
$\mathcal{M} \boxplus \mathcal{H}$ is compact.
\item[3)] the cardinality of $\mathrm{Ver}(\mathcal{D}_{\nu R} [\mathcal{M} \boxplus \mathcal{H}] )$ is
\[
\sum_{\sigma \in \mathrm{Ver}( \mathcal{D}_R [ \mathcal{H}] )}
 v^{[0, \sigma(0)]}
 v^{[1, \sigma(1)]}
 \cdots
 v^{[R-1, \sigma(R-1)]},
\]
where $v^{[r, \sigma(r)]}$ denotes
 the cardinality of $\mathrm{Ver}( \mathcal{M}^{[r, \sigma(r)]} )$.
\end{itemize}
\end{theorem}

For presenting the proof of the theorem above (Theorem \ref{hagiwaraPermutationCodes2012:thm:consolidationConstruction}),
let us prepare some terminologies and statements.
\begin{definition}[subtotal]
Let $\mathcal{M} \boxplus \mathcal{H}$ be a consolidation
for a set $\{ \mathcal{M}^{[r_0,r_1]} \}$ of quasi-homogeneous constraints 
for a $\nu$-by-$\nu$ matrix and 
a doubly stochastic constraint $\mathcal{H}$
for an $R$-by-$R$ matrix.

For $X \in \mathcal{D}_{\nu R}[ \mathcal{M} \boxplus \mathcal{H} ]$,
define an $R$-by-$R$ matrix $H$ as
\[
H_{r_0, r_1} := \sum_{0 \le i < \nu} X_{i, 0}^{[r_0, r_1]}.
\]
We call $H$ a \textbf{subtotal} of $X$.
\end{definition}

\begin{lemma}
\label{hagiwaraPerm2012:lem:subtotal}
Let $\mathcal{M} \boxplus \mathcal{H}$ be a consolidation
for
a set
$\{ \mathcal{M}^{[r_0,r_1]} \}$
of
quasi-homogeneous constraints
for a $\nu$-by-$\nu$ matrix and 
a doubly stochastic constraint $\mathcal{H}$
for an $R$-by-$R$ matrix.
Let $X \in \mathcal{D}_{\nu R} [ \mathcal{M} \boxplus \mathcal{H} ]$,
and let $H$ be a subtotal of $X$.

Then $H \in \mathcal{D}_R [ \mathcal{H} ]$ holds.
\end{lemma}
\begin{proof}
Since $X \in \mathcal{D}_{\nu R}[ \mathcal{M} \boxplus \mathcal{H} ]$,
$X \models h^{\#}$ holds for any $h \in \mathcal{H}$.
Writing a constraint $h(H)$ as
\[h(H) : \sum_{0 \le r_0, r_1 < R} c_{r_0, r_1} H_{r_0, r_1} \trianglerighteq c_0,\]
we have 
$\sum_{0 \le r_{0}, r_{1} < R}
 c_{r_0,r_1} 
 (\sum_{0 \le j < \nu} X_{0,j}^{[r_0,r_1]})
 \trianglerighteq c_0 $ holds.
By the definition of subtotal,
$
\sum_{0 \le r_{0}, r_{1} < R}
 c_{r_0,r_1} H_{r_0,r_1} \trianglerighteq c_0$ holds.
Hence $H \models h$.
This implies $H \in \mathcal{D}_R [ \mathcal{H} ]$.
\end{proof}

\begin{lemma}
\label{hagiwaraPerm2012:lem:consolidationIsDSC}
Let $\mathcal{M}^{[r_0,r_1]}$ be a quasi-homogeneous
 constraint for a $\nu$-by-$\nu$ matrix
 for $0 \le r_0, r_1 < R$.
Let $\mathcal{H}$ be a doubly stochastic constraint for an $R$-by-$R$ matrix.

Then the consolidation
$\mathcal{M} \boxplus \mathcal{H}$ is a doubly stochastic constraint
 for a $\nu R$-by-$\nu R$ matrix.
\end{lemma}
\begin{proof}
Let us show that the definition of linear constraint holds on
 $\mathcal{M} \boxplus \mathcal{H}$.
Since $\square$ and $\#$ map a linear constraint to a linear constraint,
$\mathcal{M} \boxplus \mathcal{H}$ is a set of linear constraints.

Let $X$ be a matrix such that $X \models \mathcal{M} \boxplus \mathcal{H}$.

(On row-sum constraints):
For $0 \le r_0 < R$ and $0 \le i < \nu$,
\begin{eqnarray*}
\text{$(r_0 \nu+i)$th row-sum}
&=&
\sum_{0 \le r_1 \nu + j < \nu R} X_{r_0 \nu + i, r_1 \nu + j}\\
&=&
\sum_{0 \le r_1 < R} \sum_{ 0 \le j < \nu} X^{ [r_0, r_1] }_{i,j}.
\end{eqnarray*}
By the weak row-sum constraints,
\begin{eqnarray*}
\sum_{0 \le r_1 < R} \sum_{ 0 \le j < \nu} X^{[r_0,r_1]}_{i,j}
&=&
\sum_{0 \le r_1 < R} \sum_{ 0 \le i < \nu} X^{[r_0,r_1]}_{ i, 0}\\
&=&
\sum_{0 \le r_1 < R} H_{r_0, r_1}.
\end{eqnarray*}
where $H$ is the subtotal of $X$.
The last term is the $r_0$th row-sum of $H$.
Remember that 
$\mathcal{H}$ is a doubly stochastic constraint
and
$H \models \mathcal{H}$,
in particular,
$H$ satisfies the row-sum constraints.
Thus the $r_0$th row-sum is equal to $1$.

(On column-sum constraints):
A proof is done by a similar argument to the above
and by the constraint sum property (Lemma \ref{hagiwaraPermutationCodes2012:proposition:constantSumProperty}).

(On positivity):
Positivity holds from one of $\mathcal{M}^{[r_0, r_1]}$.
\end{proof}

\begin{lemma}
\label{hagiwaraPerm2012:lem:local}
Let $\mathcal{M} \boxplus \mathcal{H}$ be a consolidation
for
a set 
$\{ \mathcal{M}^{[r_0,r_1]} \}$
of
quasi-homogeneous constraints
for a $\nu$-by-$\nu$ matrix and 
a doubly stochastic constraint $\mathcal{H}$
for an $R$-by-$R$ matrix.
Let $X \in \mathcal{D}_{\nu R} [ \mathcal{M} \boxplus \mathcal{H} ]$,
and let $H$ be a subtotal of $X$.

If an entry $H_{r_0, r_1}$ is not equal to $0$,
$\frac{1}{H_{r_0,r_1}} X^{[r_0,r_1]} \in \mathcal{D}_{R} [ \mathcal{M}^{[r_0, r_1]} ]$
holds,
where $X^{[r_0, r_1]}$ is the $(r_0, r_1)$th block of $X$.
\end{lemma}
\begin{proof}
Let $m \in \mathcal{M}^{[r_0, r_1]}$
and 
\[m(Y) : \sum_{0 \le i,j < \nu} c_{i,j} Y_{i,j} \trianglerighteq c_0.\]
Since $\mathcal{M}^{[r_0, r_1]}$ is quasi-homogeneous,
we have $c_0 =0$ or $c_0=1$.

Case $c_0 = 0$:
since $X \in \mathcal{D}_{\nu R} [ \mathcal{M} \boxplus \mathcal{H} ]$
and
the merged constraint $m^{\square}$ is homogeneous,
$X \models m^{\square}$.
Thus
\[
\sum_{0 \le i,j < \nu} c_{i,j} X^{[r_0,r_1]}_{i,j} \trianglerighteq 0
\] holds.
By dividing it by $H_{r_0, r_1}$,
we have
\[
\sum_{0 \le i,j < \nu} c_{i,j} ( \frac{1}{H_{r_0, r_1}} X^{[r_0,r_1]}_{i,j}) \trianglerighteq 0.
\]
Hence $( \frac{1}{H_{r_0, r_1}} X^{[r_0,r_1]}_{i,j} ) \models m$.

Case $c_0 = 1$:
By the definition of quasi-homogeneous constraint,
$m$ must be a row-sum or the column-sum constraint.
By the constant sum property (Lemma \ref{hagiwaraPermutationCodes2012:proposition:constantSumProperty}),
any row-sum and any column-sum of $(\frac{1}{H_{r_0, r_1}} X^{[r_0,r_1]}_{i,j})$ are equal.
Therefore it is enough to show that
\[ \sum_{0 \le i < \nu} \frac{1}{H_{r_0, r_1} } X^{[r_0,r_1]}_{i,0} = 1, \]
equivalently
\[ \sum_{0 \le i < \nu} X^{[r_0,r_1]}_{i,0} = H_{r_0, r_1}. \]
This follows from the definition of subtotal $H$.
\end{proof}

\begin{lemma}
\label{hagiwaraPermutationCodes2012:lemma:supportStructure}
Let $X, X^{(0)},$ and $X^{(1)}$ be doubly stochastic matrices
 different from each other
and $c_0, c_1$ positive numbers
such that $X = c_0 X^{(0)} + c_1 X^{(1)}$ holds.

If $X_{i,j}^{(0)} \neq 0$ or $X_{i,j}^{(1)} \neq 0$ hold,
 then $X_{i,j} \neq 0$ holds.
\end{lemma}
\begin{proof}
It follows immediately from the positivity of doubly stochastic matrices.
\end{proof}

\begin{proof}[Proof for Theorem \ref{hagiwaraPermutationCodes2012:thm:consolidationConstruction}]
1)
Let $H$ be a subtotal of $X$.
First let us show that $X \in \mathrm{Ver}( \mathcal{D}_{\nu R}[ \mathcal{M} \boxplus \mathcal{H} ])$
implies $H \in \mathrm{Ver}( \mathcal{D}_{R} [ \mathcal{H} ])$.
By Lemma \ref{hagiwaraPerm2012:lem:consolidationIsDSC}, $H \in \mathrm{D}_R [\mathcal{H}]$ holds.
Assume $H \not \in \mathrm{Ver}( \mathcal{D}_{R} [ \mathcal{H} ])$,
in other words, $H = c_0 H^{(0)} + c_1 H^{(1)}$ 
for some $c_0, c_1 > 0$ and $H^{(0)}, H^{(1)} \in \mathcal{D}_R [ \mathcal{H} ]$
with $H^{(0)} \neq H^{(1)}$.
This implies that $H_{r_0, r_1} = c_{0} H^{(0)}_{r_0, r_1} +  c_{1} H^{(1)}_{r_0, r_1}.$
Define matrices $X^{(0)}$ and $X^{(1)}$ by
\begin{equation*}
X^{(s) [r_0,r_1]} := \left\{
\begin{array}{ll}
 0 & \text{if } H_{r_0, r_1} = 0 \\
 \frac{ H^{(s)}_{r_0,r_1} }{H_{r_0, r_1} }X^{[r_0, r_1]} & \text{if } H_{r_0, r_1} \neq 0
\end{array}
\right.
\end{equation*}
for $s=0, 1$.
Then we have $X = c_0 X^{(0)} + c_1 X^{(1)}$ and $X^{(0)}, X^{(1)} \in \mathcal{D}_{\nu R}[ \mathcal{M} \boxplus \mathcal{H} ])$
with $X^{(0)} \neq X^{(1)}$.
Hence $X \not \in \mathrm{Ver}( \mathcal{D}_{\nu R}[ \mathcal{M} \boxplus \mathcal{H} ])$.

Next,
let us show that
 $X \in \mathrm{Ver}( \mathcal{D}_{\nu R}[ \mathcal{M} \boxplus \mathcal{H} ])$
and $H_{r_0, r_1} \neq 0$
implies $\frac{1}{H_{r_0, r_1} }X^{[r_0, r_1]} \in \mathrm{Ver}( \mathcal{D}_{\nu} [ \mathcal{M}^{[r_0, r_1]} ])$.
By Lemma \ref{hagiwaraPerm2012:lem:local},
$\frac{1}{H_{r_0, r_1} }X^{[r_0, r_1]} \in \mathcal{D}_{\nu} [ \mathcal{M}^{[r_0, r_1]} ]$ holds.
Assume that $\frac{1}{H_{r_0, r_1} }X^{[r_0, r_1]} \not \in \mathrm{Ver}( \mathcal{D}_{\nu} [ \mathcal{M}^{[r_0, r_1]} ])$.
In other words, $\frac{1}{H_{r_0, r_1} }X^{[r_0, r_1]} = c_2 Y^{(2)} + c_3 Y^{(3)}$ 
for some $c_2, c_3 > 0$ and $Y^{(2)}, Y^{(3)} \in \mathcal{D}_{\nu} [ \mathcal{M}^{[r_0, r_1]} ]$
with $Y^{(2)} \neq Y^{(3)}$.
Define matrices $X^{(2)}, X^{(3)}$ of size $\nu R$-by-$\nu R$ as
\begin{equation*}
X^{(s) [r_2,r_3]} := \left\{
\begin{array}{ll}
 Y^{(s)} & \text{if } r_2 = r_0, r_3 = r_1 \\
 X^{[r_2, r_3]} & \text{otherwise}
\end{array}
\right.
\end{equation*}
for $s=2, 3$.
Then we have 
$\frac{1}{H_{r_0, r_1} }X^{[r_0, r_1]}  = c_2 X^{(2)} + c_3 X^{(3)}$
and
$X^{(2)}, X^{(3)} \in \mathcal{D}_{\nu R}[ \mathcal{M} \boxplus \mathcal{H} ])$
with $X^{(2)} \neq X^{(3)}$.
Hence $X \not \in \mathrm{Ver}( \mathcal{D}_{\nu R}[ \mathcal{M} \boxplus \mathcal{H} ])$.

Conversely,
let $H' \in \mathrm{Ver}( \mathcal{D}_{R}[ \mathcal{H} ])$
and $X'^{[r_0, r_1]} \in \mathrm{Ver}( \mathcal{D}_{\nu}[ \mathcal{M}^{[r_0, r_1]} ])$.
Then define $X' := ( H_{r_0, r_1} X'^{[r_0, r_1]})$.
By the compact property,
$H$ and $X'^{[r_0, r_1]}$ are permutation matrices.
This implies that $X'$ is a permutation matrix.

We claim that $X' \in \mathrm{Ver}(\mathcal{D}_{\nu R}[ \mathcal{M} \boxplus \mathcal{H} ] )$.
If not, there exist $c_4, c_5 > 0$
 and $X^{(4)}, X^{(5)} \in \mathcal{D}_{\nu R}[ \mathcal{M} \boxplus \mathcal{H} ]$
such that 
\[
X' = c_4 X^{(4)} + c_5 X^{(5)}, X^{(4)} \neq X^{(5)}.
\]
Since $\mathcal{D}_{\nu R} [ \mathcal{M} \boxplus \mathcal{H} ]
\subset
\mathrm{DSM}_{\nu R}$,
we have $X', X^{(4)}, X^{(5)} \in \mathrm{DSM}_{\nu R}$.
Thus $X \not \in \mathrm{Ver}( \mathrm{DSM}_{\nu R})$,
but $X'$ is a permutation matrix.
This contradicts the Birkhoff von-Neumann theorem
(Theorem \ref{compactGraphCopyTheory:theorem:birkhoffVon-Neumann}).

2)
By Lemma \ref{hagiwaraPerm2012:lem:consolidationIsDSC},
the consolidation $ \mathcal{M} \boxplus \mathcal{H}$ is
a doubly stochastic constraint.
$ \mathcal{M} \boxplus \mathcal{H}$ consists of finite linear constraints.
Since $\mathrm{DSM}_{\nu R}$ is bounded and
$\mathcal{D}_{\nu R}[ \mathcal{M} \boxplus \mathcal{H} ] \subset
 \mathrm{DSM}_{\nu R}$,
$\mathcal{D}_{\nu R}[ \mathcal{M} \boxplus \mathcal{H} ] $ is bounded.

It remains to show that $\mathrm{Ver}( \mathcal{D}_{\nu R}[ \mathcal{M} \boxplus \mathcal{H} ] ) \subset S_{\nu R}$.
By 1), any vertex is written as $(H_{r_0, r_1} X^{[r_0, r_1]})$.
Since $X^{[r_0, r_1]}$
and $H$ are permutation matrices,
$(H_{r_0, r_1} X^{[r_0, r_1]})$
is a permutation matrix,
\textit{i.e.},
$\mathrm{Ver}( \mathcal{D}_{\nu R}[ \mathcal{M} \boxplus \mathcal{H} ] ) \subset S_{\nu R}$.

Hence $\mathcal{M} \boxplus \mathcal{H}$ is compact.

3)
Let $H \in \mathrm{Ver}[ \mathcal{D}_{R} (\mathcal{H}) ]$.
Define $V_H := \{ X \in \mathrm{Ver}( \mathcal{D}_{\nu R}[ \mathcal{M} \boxplus \mathcal{H} ] )
 \mid \text{the subtotal of } X \text{ is } H \}$.
Let us consider the cardinality of $V_H$.
Since $\mathcal{H}$ is compact,
$H$ is a permutation matrix.
Let $\sigma$ denote the permutation associated with $H$.
As a direct corollary of 1), the cardinality $\# V_H$ is at least 
$v^{[0, \sigma(0)]} v^{[1, \sigma(1)]} \dots v^{[R-1, \sigma(R-1)]}$.

Since 
$H$ is a permutation matrix,
 for each $0 \le r_0 < R$,
 there exists a unique $0 \le r_1 < R$ such that
 the $(r_0, r_1)$th block $X^{[r_0, r_1]}$ is not a zero-matrix,
that is $r_1 = \sigma(r_0)$.
Furthermore, $X^{[r_0, \sigma(r_0) ]} \in \mathrm{Ver}( \mathcal{D}_{\nu R}[ \mathcal{M}^{[r_0, \sigma(r_0)]} ] )$.
Thus $\# V_H $ is at most $v^{[0, \sigma(0)]} v^{[1, \sigma(1)]} \dots v^{[R-1, \sigma(R-1)]}$.
This implies that
$\# V_H = v^{[0, \sigma(0)]} v^{[1, \sigma(1)]} \dots v^{[R-1, \sigma(R-1)]}$.

Summing up $\# V_H$ so that $H$ is taken over $\mathrm{Ver}( \mathcal{D}_{R} \mathcal{H} )$,
we conclude the proof.
\end{proof}

\section{Encoding and Decoding for Permutation Code from Consolidation}
\subsection{Structure of Vertex of Consolidation}
In this section, we discuss an encoding and an decoding algorithms
for a permutation code $( G_{\mathcal{M} \boxplus \mathcal{H}}, \mu)$
associated with a consolidation $\mathcal{M} \boxplus \mathcal{H}$
for a set $\{ \mathcal{M}^{[r_0, r_1]} \}$ of 
compact quasi-homogeneous constraints for $\nu$-by-$\nu$ matrix
and a compact doubly stochastic constraint $\mathcal{H}$.

We assume that
``quasi-homogeneous constraints are the same
if their 1st indices are the same, \textit{i.e.},
for any $0 \le r_0, r_1 < R$,
 $\mathcal{M}^{[i, r_1]} = \mathcal{M}^{[i,0]}$ holds.
It implies that $v^{[i, r_1]} = v^{[i,0]}$ for $0 \le r_0 < R$,
where $v^{[r_0, r_1]} = \# \mathrm{Ver}( \mathcal{M}^{[r_0, r_1]}  )$.
Let $v_{r_0} $ denote $v^{[r_0, 0]}$ for $0 \le r_0 < R$
and $v_{R}$ denote $\# \mathrm{Ver}( \mathcal{H}  )$.
By Theorem \ref{hagiwaraPermutationCodes2012:thm:consolidationConstruction},
we have
\[
\# ( G_{ \mathcal{M} \boxplus \mathcal{H} }, \mu)  = v_0 v_1 \dots v_{R-1} v_R.
\]

\subsection{Encoding}
In sections
\ref{hagiwaraPermutationCdoes2012:subsection:CompactConstraints}
and
\ref{compactGraphCopyTheory:Section:ExampleOfCompactSeedGraph},
we gave some examples of compact quasi-homogeneous constraints.
Except for trees, we have a (reasonably) small computational cost
encoding algorithm for the matrix size.
\begin{itemize}
\item For $\mathcal{L}_D$ in Note \ref{hagiwaraPermutationCodes2012:note:doublyStochasticMatrix} for an $n$-by-$n$ matrix, 
an encoding algorithm with computational cost $O(n \log n)$ is known.
See \S 5.1 of \cite{Knuth:TheArtOfComputerProgramming}.
\item For a pure involution constraint $\mathcal{L}_P$
for an $n$-by-$n$ matrix,
an encoding algorithm with computational cost $O(n^2)$ has been given
 in \cite{Wadayama:LPDecodablePermutationCodesBasedOnLinearlyConstrainedPermutationMatrices}.
Note that it is compact only for $n= 2, 4$.
\item For a compact graph constraint $\mathcal{L}_\Gamma$
for an $n$-by-$n$ matrix,
we can construct 
an encoding algorithm with computational cost at most $O(n)$,
since the cardinality $G_\Gamma$ is at most $2n$.
\end{itemize}

From here, we define an encoding map $\mathrm{Enc}$
from
$\{ 0, 1, \dots, v_0 v_1 \dots v_R -1 \}$
to $G_{ \mathcal{M} \boxplus \mathcal{H} })$.
Our idea is to reduce this discussion to
 a ``local'' encoding $\mathrm{Enc}_{r_0} : \{0,1, \dots, v_{r_0} -1\}
\rightarrow G_{\mathcal{M}^{[r_0, 0]}}$,
 for $0 \le r_0 < R$
and
 an encoding $\mathrm{Enc}_{R} : \{0,1, \dots, v_{R} -1\}
\rightarrow G_{\mathcal{H} }$.

Remark that Theorem \ref{hagiwaraPermutationCodes2012:thm:consolidationConstruction}, 1)
can be stated as
\begin{eqnarray*}
& &\mathrm{Ver}( \mathcal{D}_{\nu R}[ \mathcal{M} \boxplus \mathcal{H} ])\\
&=&
\{
(\sigma | g_0, \dots, g_{R-1})
\mid\\
& &
\sigma \in \mathrm{Ver}( \mathcal{D}_{R}[\mathcal{H}] ),
g_r \in \mathrm{Ver}( \mathcal{D}_{\nu}[ \mathcal{M}^{[r, \sigma(r)]} ])
\}.
\end{eqnarray*}

\begin{definition}[Encoding algorithm $\mathrm{Enc}$]
We define an encoding algorithm $\mathrm{Enc}$ by the following steps:
\begin{itemize}
\item Input: integer $0 \le \mathrm{mes} < v_0 v_1 \dots v_R$ and a vector $\mu \in \mathbb{R}^{\nu R}$.
\item Output: a vector $\mu' \in \mathbb{R}^{\nu R}$.
\item 1. Set $j := 0$.
\item 2. Set $\mathrm{mes}_j := i \pmod {v_0 v_1 \dots v_{R-1 -j}}$.
\item 3. Update $\mathrm{mes} := (\mathrm{mes}-\mathrm{mes}_j) \; \mathrm{div} \; v_j$.
\item 4. Encode $\mathrm{mes}_j$ to a permutation $g_j$ by $\mathrm{Enc}_j$.
\item 5. Update $j := j+1$.
\item 6. If $j=R+1$ then go to 7. Else go to 2.
\item 7. Set $\mu' := (g_R | g_0, g_1, \dots, g_{R-1}) \mu$,
where $(g_R | g_0, g_1, \dots, g_{R-1})$ is a permutation matrix\footnote{
This notation is defined in \S\ref{hagiwaraPermutationCodes2011:definition:wreathProduct}.}.
\item 8. Output $\mu'$.
\end{itemize}
\end{definition}

\subsection{Decoding to Integer}
Similar to the last subsection,
we present a message decoding algorithm
by using a local decoding $\mathrm{Dec}_i$,
where $\mathrm{Dec}_{i}$ for $0 \le i < R$ is a decoder for $\mathcal{M}^{[i, 0]}$
and $\mathrm{Dec}_R$ is a decoder for $\mathcal{H}$.
\begin{definition}[Message Decoding Algorithm $\mathrm{Dec}$]
We define a message decoding algorithm $\mathrm{Dec}$ by the following steps:
\begin{itemize}
\item Input: a permutation matrix $X$.
\item Output: an integer $\mathrm{mes}$.
\item 1. Divide $X$ into $R^2$-blocks $X^{[r_0, r_1]}$
See \S\ref{hagiwaraPerm2012:subsec:MergedConstaints}
\item 2. Calculate the subtotal $H$ of $X$.
\item 3. Set $\mathrm{mes} := \mathrm{Dec}_R (H)$.
\item 4. Set $i := R$.
\item 5. Update $\mathrm{mes} := \mathrm{mes} \times v_i$.
\item 6. Check if there exists 
a unique $0 \le j <R$ such that $X^{[i, j]} \neq 0$.
If not, terminate the algorithm.
\item 7. Update $\mathrm{mes} := \mathrm{mes} + \mathrm{Dec}_i (X^{[i,j]})$.
\item 8. Update $i := i-1$.
\item 9. If $i \ge 0$, go to 5.
\item 10. Output $\mathrm{mes}$.
\end{itemize}
\end{definition}

\section{Automorphism of a Union of Seed Graphs}
We present an application of our theory.
For stating one of our main results,
we introduce more terminologies on graphs.
For a graph $\Gamma = ( V, E) $ and its vertex $v \in V$,
the cardinality $\# \{ i \in V \mid (i, v) \in E \}$ is said to be \textbf{in-degree} of $v$.
Similarly,
the cardinality $\# \{ j \in V \mid (v, j) \in E \}$ is said to be \textbf{out-degree} of $v$.

\begin{definition}[Seed Graph]
Let $\Gamma$ be a (directed or un-directed) graph.
In this paper,
$\Gamma$ is called a \textbf{seed graph}
if
$\Gamma$ is connected
and the in-degree of $v$ is equal to the out-degree of $v$ for any vertex $v$ of $\Gamma$.
\end{definition}

\begin{example}
Any un-directed connected graph is a seed graph,
\textit{e.g.}, tree, circle, and perfect graph.
As an example of directed graph,
a cycle is a seed graph
(see \S\ref{compactGraphCopyTheory:Section:ExampleOfCompactSeedGraph}).

A graph whose adjacency matrix is an identity matrix is not a seed graph,
since it is not connected.
\end{example}

For a graph $\Gamma = (\{ 0, 1,\dots, \nu \}, E)$ and a positive integer $R$,
a graph $\Gamma^{(R)} =(\{0,1,\dots, \nu R -1 \}, E^{(R)})$ is defined by the following adjacency matrix $A^{\Gamma^{(R)}}$:
\begin{eqnarray*}
A^{\Gamma^{(R)}} :=
\left(
\begin{array}{cccc}
A^{\Gamma} & \mathbf{0} & \dots      & \mathbf{0} \\
\mathbf{0} & A^{\Gamma} & \dots      & \mathbf{0} \\
\vdots     &            & \ddots     & \vdots     \\
\mathbf{0} & \dots      & \mathbf{0} & A^{\Gamma} \\
\end{array}
\right),
\end{eqnarray*}
where $A^{\Gamma}$ is the adjacency matrix for $\Gamma$ and $\mathbf{0}$ is a zero-matrix.
In this paper, the graph $\Gamma^{(R)}$ is said to be a \textbf{union} of $R$-$\Gamma$s.
Note that $\Gamma^{(R)}$ is not connected.
This implies that $\Gamma^{(R)}$ is not a seed graph.

The main contribution of this section is the following:
\begin{theorem}\label{compactGraphCopyTheory:theorem:main}
Let $\Gamma$ be a seed graph.

If $\Gamma$ is compact, then a union $\Gamma^{(R)}$ is also compact
for any positive integer $R$.
Furthermore
\[
\mathrm{Ver} ( \mathcal{D}[ \mathcal{L}_{\Gamma^{(R)} } ] )
=
\mathrm{Aut}(\Gamma^{(R)}).
\]
\qed
\end{theorem}

\begin{remark}
\label{compactGraphCopyTheory:remark:doubleCopyTheorem}
If $\Gamma$ is ``un-directed and $R=2$,''
 Theorem \ref{compactGraphCopyTheory:theorem:main} is the same as Tinhofer's theorem
\cite{Tinhofer:aNoteOnCompactGraphs}.
\end{remark}

The following lemma is a slight generalization of Tinhofer's lemma \cite{Tinhofer:graphIsomorphismAndTheoremsOfBirkhoffType}:
\begin{lemma}[Generalized Tinhofer's Lemma ]
\label{compactGraphCopyTheory:lemma:GeneralizedTinhoffer}
Let $\Gamma$ be a seed graph with $n$-vertices
 and $X$ a doubly stochastic matrix of size $n$-by-$n$
such that $A^{\Gamma} X = X A^{\Gamma}$ holds, where $A^{\Gamma}$ is the adjacency matrix of $\Gamma$.

For $0 \le i,j < n$,
if the in-degrees of $i$ and $j$ are different,
then $X_{i,j} = 0$ holds.
\end{lemma}
The proof is given by almost the same argument as Tinhofer's original one.
Please refer to \cite{Tinhofer:graphIsomorphismAndTheoremsOfBirkhoffType}.

\begin{lemma}\label{compactGraphCopyTheory:lemma:KeyLemmaForMainTheorem}
Let $\Gamma = (\{ 0, 1,\dots, \nu-1 \}, E)$ be a seed graph.
Let $X \in \mathcal{D}_{\nu R} [ \mathcal{L}_{ \Gamma^{(R)} }]$.

For any $0 \le r_0, r_1 < R$,
$X^{[r_0, r_1]}$ satisfies weak row-constant constraints
and 
weak column-sum constraints.
\end{lemma}
\begin{proof}
If $X^{[r_0, r_1]}$ is a zero-matrix,
$X^{[r_0, r_1]}$ satisfies week row-constraints and week column-constraints..

From here, let us assume $X^{[r_0, r_1]}$ is not a zero-matrix.

This proof is an analogue of 
Tinhofer's argument \cite[Theorem 5]{Tinhofer:aNoteOnCompactGraphs}.
Let $e = (1,1,\dots,1)$ be the all 1 vector of length $n$.
For a matrix $M$, let $c_{M}$ (resp. $r_M$) denote $eM$ (resp. $M e^T$),
\textit{i.e.}, $c_M$ (resp. $r_M$) is the vector consists of the column (resp. row) sum of $M$.
We shall show that $c_{X^{[r_0, r_1]}}$ and $r_{X^{[r_0, r_1]}}$ are constant vectors.

By the assumption on $X$ and $A^{\Gamma}$, the following holds:
\[
 e X^{[r_0, r_1]} A^{\Gamma} = e A^{\Gamma} X^{[r_0, r_1]}.
\]
L.H.S. $c_{X} A^{\Gamma}$is equal to
\[
 ( \sum_{0 \le h < \nu} c_{X^{[r_0, r_1]}, h} A^{\Gamma}_{h 1},
 \dots, \sum_{0 \le h < \nu} c_{X^{[r_0, r_1]}, h} A^{\Gamma}_{h n} ),
\]
where $c_{X^{[r_0, r_1]}, h}$ is the $h$th element of $c_X^{[r_0, r_1]}$.
On the other hand, R.H.S. $c_{A^{\Gamma}} X^{[r_0, r_1]}$ is equal to
\[
 (
 \sum_{0 \le h < \nu} c_{ A^{\Gamma}, h} X^{[r_0, r_1]}_{h1}, \dots, \sum_{0 \le h < \nu} c_{ A^{\Gamma}, h} X^{[r_0, r_1]}_{hn} ).
\]
By a generalized Tinhofer's lemma (Lemma\ref{compactGraphCopyTheory:lemma:GeneralizedTinhoffer}),
$\sum_{0 \le h < \nu} c_{ A^{\Gamma}, h} X^{[r_0, r_1]}_{hj} = c_{A^{\Gamma} , j} c_{X^{[r_0, r_1]}, j}$ holds for $ 0 \le j < \nu$.
Therefore
\begin{eqnarray*}
\text{$j$th element of L.H.S.}
&=&\\
 \sum_{0 \le h < \nu} c_{X^{[r_0, r_1]}, h} A_{\Gamma, hj}
&=&
 c_{A^{\Gamma}, j}  c_{X^{[r_0, r_1]}, j} \\
&=&
\text{$j$th element of R.H.S.}
\end{eqnarray*}

Since $\Gamma$ is connected, $c_{A^{\Gamma}, j} \neq 0$.
By dividing $j$th element with $c_{A^{\Gamma}, j}$,
we have
\[
 \sum_{0 \le h < \nu} c_{X^{[r_0, r_1]}, h} \frac{A^{\Gamma}_{hj}}{ c_{A^{\Gamma}, j} }
 = c_{X^{[r_0, r_1]}, j}.
\]
By using the notation $S := ( \frac{A^{\Gamma}_{ij}}{ c_{A^{\Gamma}, j} } )$,
the equation above implies
\[
 c_{X^{[r_0, r_1]}} S = c_{X^{[r_0, r_1]}}, \text{ equivalently, }
 S^T c_{X^{[r_0, r_1]}}^T = c_{X^{[r_0, r_1]}}^T.
\]
Therefore $c_{X^{[r_0, r_1]}}^T$ is an eigenvector of $S^T$ of eigenvalue $1$.
By the construction of $S$, $S^T$ is a stochastic matrix.
Remember $\Gamma$ is connected.
Therefore
$c_{X^{[r_0, r_1]}}$ is a constant vector,
by general linear algebra and probabilistic theory.

Hence $c_{X^{[r_0, r_1]}} = (c, c, \dots, c)$ for some $c > 0$.
By a similar discussion, $r_{X^{[r_0, r_1]}} = (r, r, \dots, r)^T$ for some $r > 0$.
\end{proof}

\begin{corollary}
\label{hagiwaraPerm2012:cor:colocolo}
Let $\Gamma$ be a compact seed graph with $\nu$-vertices
and $\mathcal{L}_{D^{(R)}}$ a doubly stochastic constraint
in Note \ref{hagiwaraPermutationCodes2012:note:doublyStochasticMatrix},
for an $R$-by-$R$ matrix.
Let $\mathcal{M}_\Gamma$ be a set $\{ \mathcal{M}^{[r_0, r_1]} \}$
of quasi-homogeneous constrains,
where 
$\mathcal{M}^{[r_0, r_1]} := \mathcal{L}_{\Gamma}$.

Then we have
\[
\mathcal{D}[ \mathcal{L}_{\Gamma^{(R)}} ]
=
\mathcal{D}[ \mathcal{L}_{\Gamma} \boxplus \mathcal{L}_{D^{(R)}} ]
\]
\end{corollary}
\begin{proof}
By using the block component $X^{ [ij ]}$,
the equation $ X A^{\Gamma^{(R)}} = A^{\Gamma^{(R)}} X$
is equivalent to
$ X^{[ij]} A^{\Gamma} = A^{\Gamma} X^{[ij]},
$ for all $0 \le i, j < R.$
Therefore $\mathcal{L}_{\Gamma^{(R)}} \subset
 \mathcal{M}_{\Gamma} \boxplus \mathcal{L}_{D^{(R)} }$.
It implies $\mathcal{D}[ \mathcal{L}_{\Gamma^{(R)}} ] \supset \mathcal{D}[ \mathcal{M}_{\Gamma} \boxplus \mathcal{L}_{D^{(R)} } ]$.

From here, we show that  
$X \in \mathcal{D}[ \mathcal{M}_{\Gamma} \boxplus \mathcal{L}_{D^{(R)} } ]$
for any $X \in \mathcal{D}[ \mathcal{L}_{\Gamma^{(R)}} ]$.
Note that
$\mathcal{M}_{\Gamma} \boxplus \mathcal{L}_{D^{(R)} } \setminus
\mathcal{L}_{\Gamma^{(R)}}$ consists of weak row-sum constraints
and weak column-sum constraints.
By Lemma\ref{compactGraphCopyTheory:lemma:KeyLemmaForMainTheorem},
if $X \in \mathcal{D}[ \mathcal{L}_{\Gamma^{(R)}} ]$ satisfies
weak-row constraints,
then $X \models \mathcal{M}_{\Gamma} \boxplus \mathcal{L}_{D^{(R)} }$,
in other words,
$X \in \mathcal{D}[ \mathcal{M}_{\Gamma} \boxplus \mathcal{L}_{D^{(R)} } ]$.
\end{proof}

For a set $G$ of permutation matrices and elements $(\sigma | g_0, g_1, \dots, g_{R-1}),
 (\tau | h_0, h_1, \dots, h_{R-1}) \in G \wr S_R$,
the matrix product of them satisfies
\begin{eqnarray*}
& &(\sigma | g_0, g_1, \dots, g_{R-1}) (\tau | h_0, h_1, \dots, h_{R-1}) \\
&=&
(\sigma \tau | g_{0} h_{\tau(0)}, g_{1} h_{\tau(1)}, \dots, g_{R-1} h_{\tau(R-1)}).
\end{eqnarray*}
Thus if $G$ is a group, a wreath product $G \wr S_R$ is a group too.
Let $\Gamma = (\{ 0, 1,\dots,n-1 \}, E)$ be a connected graph
and $\mathrm{Aut}( \Gamma )$ the set of automorphisms of $\Gamma$.
Since $\mathrm{Aut}(\Gamma)$ is a group and 
an automorphism is a permutation on the vertex set $\{ 0, 1, \dots, n-1 \}$,
$\mathrm{Aut}( \Gamma )$ is regarded as a subgroup of the set of $n$-by-$n$ permutation matrices.

Now we state the following:
\begin{lemma}\label{compactGraphCopyTheory:lemma:wreathProductAuto}
Let $\Gamma = (\{ 0, 1,\dots, n-1 \}, E)$ be a connected graph and $\mathrm{Aut}(\Gamma)$ be the automorphism group of $\Gamma$.
For any positive integer $R$,
$\mathrm{Aut}( \Gamma^{(R)} )$ satisfies
\[
\mathrm{Aut}( \Gamma^{(R)} ) 
=
\mathrm{Aut}(\Gamma) \wr S_R,
\]
where $\Gamma^{(R)}$ is a union of $R$-$\Gamma$s.
\end{lemma}
\begin{proof}
It is trivial that the right hand side is included in the left hand side.
We show the converse inclusion relation.
Let $X \in \mathrm{Aut}( \Gamma^{(R)} )$.
Since $\Gamma$ is connected and $\Gamma^{(R)}$ is a union,
$X$ is a permutation on connected components of $\Gamma^{(R)}$.
It implies that 
$X$ is written in a form $(\sigma | X_1, X_2, \dots, X_R)$ by using some
permutations $X_i$ and a permutation $\sigma \in S_R$.
Since each connected component of $\Gamma^{(R)}$ is isomorphic to $\Gamma$,
$X_i \in \mathrm{Aut}(\Gamma)$ holds.
This implies $X \in \mathrm{Aut}(\Gamma) \wr S_R$.
\end{proof}

\begin{proof}[Proof for Theorem \ref{compactGraphCopyTheory:theorem:main}]
By Corollary \ref{hagiwaraPerm2012:cor:colocolo},
$\mathrm{Ver}( \mathcal{D}_{\nu R}[ \mathcal{L}_{\Gamma^{(R)} } ])
\subset S_{\nu R}$ holds.
By equation (\ref{hagiwaraPerm2012:eqn:graphAut}),
$\mathrm{Ver}( \mathcal{D}_{\nu R}[ \mathcal{L}_{\Gamma^{(R)} } ])
=
\mathrm{Aut}( \mathcal{L}_{\Gamma^{(R)} } )$
holds.
\end{proof}

\section{Further Discussion}
\subsection{Number of Linear Constraints}
\label{compactGraphCopyTheory:section:additionalEquationsMeansToReduceVariables}
Considering the computational cost of linear programming,
the reader may be anxious about the additional cost due to the equation $X A^\Gamma = A^\Gamma X$.
However, the additional equations may not increase the computational cost.
Conversely, it may decrease the cost,
by reducing the number of variables.

Here we present an example with a graph $\Gamma$ of type televis
(see Example \ref{hagiwaraPerm2012:expl:lineTelevis}).
Let us observe the doubly stochastic constraint $\mathcal{L}_{\Gamma^{(R)}}$
associated with an $R$-union graph $\Gamma^{(R)}$.
It is a doubly stochastic constraint for a $2R$-by-$2R$ matrix.
By writing the variable matrix as $X$,
$\mathcal{L}_{\Gamma^{(R)} }$ consists of
\begin{eqnarray*}
\sum_{0 \le i < 2R} X_{i, j_0} &=& 1, \text{ for } 0 \le j_0 < 2R,\\
\sum_{0 \le j < 2R} X_{i_0, j} &=& 1, \text{ for } 0 \le i_0 < 2R,\\
X_{i_0, j_0} &\ge& 0, \text{ for } 0 \le i_0, j_0 < 2R,\\
X^{[r_0, r_1]}_{0,0} - X^{[r_0, r_1]}_{1,1} &=& 0, \text{ for } 0 \le r_0, r_1 < R,\\
X^{[r_0, r_1]}_{0,1} - X^{[r_0, r_1]}_{1,0} &=& 0, \text{ for } 0 \le r_0, r_1 < R.
\end{eqnarray*}
Totally, they are $6 R^2 + 4R (= 2R+2R+(2R)^2 + R^2 + R^2 )$ linear constraints.

Our idea to reduce the computational cost is the following:
we regard a linear constraint
$X^{[r_0, r_1]}_{0,0} - X^{[r_0, r_1]}_{1,1} = 0$
as 
 substitution
$X^{[r_0, r_1]}_{0,0} = X^{[r_0, r_1]}_{1,1}$.
Let $Y^{[r_0, r_1]}_0$ denote $X^{[r_0, r_1]}_{0,0}$
and $X^{[r_0, r_1]}_{1,1}$.
Similarly 
Let $Y^{[r_0, r_1]}_1$ denote $X^{[r_0, r_1]}_{0,1}$
and $X^{[r_0, r_1]}_{1,0}$.
Then the number of variables are reduced from $4 R^2$ to $2 R^2$.
Furthermore, the linear constraint $\mathcal{L}_{\Gamma^{(R)}}$ is converted to
a doubly stochastic matrix $\mathcal{L}'_{\Gamma^{(R)}}$.
\begin{eqnarray*}
\sum_{0 \le r < R} Y^{r, r_1}_0 + Y^{r, r_1}_1  &=& 1, \text{ for } 0 \le r_1 < R,\\
\sum_{0 \le j < R} Y^{r_0, r}_0 + Y^{r_0, r}_1 &=& 1, \text{ for } 0 \le r_0 < R,\\
Y^{r_0, r_1}_{0}, Y^{r_0, r_1}_{1} &\ge& 0, \text{ for } 0 \le r_0, r_1 < R.
\end{eqnarray*}
Totally, they are only $2 R^2 + 2R (= R+R+2 R^2 )$ linear constraints.
Remember that
$ \mathcal{L}'_{\Gamma^{(R)}} \supset \mathcal{L}_{D^{(2R)}}$.
However, we have $\# \mathcal{L}'_{\Gamma^{(R)}} < \# \mathcal{L}_{D^{(2R)}}$,
since
$\mathcal{L}_{D^{(2R)}}$ consists of $4 R^2 + 4R$ constraints.

In this case with $R$-union televis $\Gamma^{(R)}$, the linear programming problem for error-correction is
to maximize the following value
\begin{eqnarray*}
& & \sum_{0 \le r_0, r_1 < R}
(\lambda_{2 r_0} \mu_{2 r_1}
 + \lambda_{2 r_0 +1} \mu_{2 r_1+1}) Y^{[r_0, r_1]}_0\\
&+&
 \sum_{0 \le r_0, r_1 < R}
(\lambda_{2 r_0} \mu_{2 r_1 +1} + \lambda_{2 r_0+1} \mu_{2 r_1}) Y^{[r_0, r_1]}_1,
\end{eqnarray*}
where $\lambda$ is a received vector and $\mu$ is the initial vector of permutation code.

\subsection{Distance Enlargement}
In this subsection,
we discuss a conjugated permutation code.
The following is directly obtained from definitions.
\begin{theorem}
Let $\Gamma$ be a graph and
and $P^{\sigma}$ a permutation matrix associated with a permutation $\sigma$.
Then we have
\[
\mathcal{D}_{n}[ \mathcal{L}_{\sigma (\Gamma)} ] 
=
\{ P^{\sigma} X (P^{\sigma})^{-1}
 \mid X \in \mathcal{D}_{n}[ \mathcal{L}_{\Gamma} ]
\}
,
\]
\[
\mathrm{Ver}( \mathcal{D}_{n}[ \mathcal{L}_{\sigma (\Gamma)} ] )
=
\{ P^{\sigma} X (P^{\sigma})^{-1}
 \mid X \in \mathrm{Ver}( \mathcal{D}_{n}[ \mathcal{L}_{\Gamma} )] \}
\]
where $\mathcal{L}_\Gamma$ is a doubly stochastic constraint associated with $\Gamma$,
$\mathcal{D}_n [ \mathcal{L} ]$ is the doubly stochastic polytope for $\mathcal{L}$,
and $\mathrm{Ver}( \mathcal{D} )$ is the set of vertices for $\mathcal{D}$.

Hence $\sigma(\gamma)$ is compact if and only if $\Gamma$ is compact.
\end{theorem}

For a permutation group $G$,
let us consider two values $d_l(G)$ and $d_E(G)$.

The \textbf{minimum Euclidean distance} $d_E$ is defined as
\[
d_E ( G )
 := \min_{g_0, g_1 \in G, g_0 \neq g_1 } || g_0 \mu - g_1 \mu ||^2 /2,
\]
where
$\mu = (1,2,\dots, n) \in \mathbb{R}^{n}$,
\textit{i.e.}, $\mu_i = i+1$ for $0 \le i < n$.
This value is motivated by the Euclidean distance metric
which is considered in the previous researches,
\textit{e.g.},
\cite{Wadayama:LPDecodablePermutationCodesBasedOnLinearlyConstrainedPermutationMatrices}.

For a permutation group $G$,
the \textbf{minimum Kendall-tau distance} $d_l( G )$ is defined as
\[
d_l( G )
 := \min_{g_0, g_1 \in G, g_0 \neq g_1}
 \# \{ (i,j) \mid
     0 \le i<j <n, g_0 g_1^{-1} (i)> g_0 g_1^{-1}(j) \}.
\]
This value is motivated by a distance metric
which is known as Kendall-tau distance in rank modulation researches,
\textit{e.g.},
\cite{Barg:codesInPermutationsAndErrorCorrectionForRankModulation}.

In general,
$d_l ( G_\Gamma ) = d_l ( G_{\Gamma^{(R)}} )$
and
$d_E ( G_\Gamma ) = d_E ( G_{\Gamma^{(R)}} )$
hold.
However, by using a group action,
we may enlarge the distances.

Here we give an example.
Let $\Gamma$ be a televis
and $\sigma$ is a permutation on $\{0,1,2,3\}$
defined as
$\sigma(0) = 0, \sigma(1)=2, \sigma(2)=1,$ and $\sigma(3)=3$.
By routine calculation, we can verify
$d_l ( G_\Gamma ) = 1$,
$d_E ( G_\Gamma ) = 1$,
$d_l ( G_{\Gamma^{(2)} } ) = 1$,
$d_E ( G_{\Gamma^{(2)} } ) = 1$,
$d_l ( G_{\sigma( \Gamma^{(2)} ) } )=2$,
and
$d_E ( G_{\sigma( \Gamma^{(2)} ) } )= 2$.

To characterize which permutation $\sigma$ maximizes these distances
is interesting but not an easy problem.
We leave this problem as an open problem.

\section*{Acknowledgments}
The author thanks Mr.~Justin Kong, Ms.~Catherine Walker and Prof.~J.~B.~Nation for 
their valuable comments and suggestions to improve the quality of the paper.

\end{document}